\documentclass{amsart}
\usepackage{setspace, amsmath, amsthm, amssymb, amsfonts, amscd, epic, graphicx, ulem, dsfont}
\usepackage[T1]{fontenc}
\usepackage{multirow}
\usepackage{bbm}
\usepackage{enumerate}

\makeatletter \@namedef{subjclassname@2010}{
  \textup{2010} Mathematics Subject Classification}
\makeatother

\newtheorem{thm}{Theorem}[section]

\newtheorem{cor}[thm]{Corollary}

\newtheorem{pro}[thm]{Proposition}
\newtheorem{conj}[thm]{Conjecture}

\theoremstyle{remark}
\newtheorem*{rema}{Remark}

\theoremstyle{definition}
\newtheorem*{defn}{Definition}
\newtheorem{exa}[thm]{\textbf{Example}}

\newcommand{\R}{\mathbb{R}}

\newcommand{\N}{\mathbb{N}}

\newcommand{\C}{\mathbb{C}}

\begin{document}

\title[The Fuglede Theorem and Some Intertwining Relations]{The Fuglede Theorem and Some Intertwining Relations}
\author[I. F. Z. Bensaid, S. Dehimi, B. Fuglede and M. H. MORTAD]{Ikram Fatima Zohra Bensaid, Souheyb Dehimi, Bent Fuglede and Mohammed Hichem Mortad$^*$}

\dedicatory{}
\thanks{* Corresponding author.}
\date{}


\address{(The first author): Department of Mathematics, University of Cadiz. Avenida de
la Universidad s/n E-11405. Jerez de la Frontera. Spain.}

\email{ikram.bensaid@alum.uca.es}

\address{(The second author): University of Mohamed El Bachir El Ibrahimi, Bordj Bou Arreridj.
Algeria.}

\email{sohayb20091@gmail.com}

\address{(The third author): Department of Mathematical Sciences, Universitetsparken 5, 2100 Copenhagen, Danmark.}

\email{fuglede@math.ku.dk}

\address{(The corresponding author) Department of
Mathematics, University of Oran 1, Ahmed Ben Bella, B.P. 1524, El
Menouar, Oran 31000, Algeria.}

\email{mhmortad@gmail.com, mortad.hichem@univ-oran1.dz.}

\begin{abstract}In this paper, we show a new and classic version of
the celebrated Fuglede Theorem in an unbounded setting. A related
counterexample is equally presented. In the second strand of the
paper, we give a pair of a closed and self-adjoint (unbounded)
operators which is not intertwined by any (bounded or closed)
operator except the zero operator.
\end{abstract}

\maketitle

\section{Introduction}

Undoubtedly, the Fuglede Theorem is the second salient result in
Operator Theory, at least, as far as normal operators are concerned.
It has many applications. The most tremendous one is the fact that
it improves the statement of the Spectral Theorem of normal
operators. To cite only a little amount of applications of this
powerful tool, we refer readers to \cite{Alb-Spain},
\cite{Chellali-Mortad}, \cite{Gustafson-Mortad-2014-Bull-SM},
\cite{Gustafson-Mortad-2016}, \cite{halmos-book-1982},
\cite{Jung-Mortad-Stochel}, \cite{Mortad-IEOT-2009},
\cite{Mortad-ZAA-2010}, \cite{mortad-CAOT-sum-normal},
\cite{OtaSchm}, \cite{PUT},
\cite{Radjavi-Rosenthal-Roots-Normal-JMAA}, \cite{Reh} and
\cite{YANGDU}.

Recall that this theorem states that if $T\in B(H)$ and $A$ is
normal (not necessarily bounded), then
\[TA\subseteq AT\Longleftrightarrow TA^*\subseteq A^*T.\]

The problem leading to this theorem was first raised by von Neumann
in \cite{von-Neumman-historic-Fuglede} who had already established
it in a finite dimensional setting (since this is seemingly not well
documented, readers may find it in e.g. Exercise 11.3.29 in
\cite{Mortad-Oper-TH-BOOK-WSPC}). Fuglede was the first one to
answer this problem affirmatively in \cite{FUG} (a quite different
proof popped up shortly afterwards and it is due to Halmos
\cite{Halmos-Fuglede-Putnam}). Then Putnam \cite{PUT} generalized
the result to:
\[TA\subseteq BT\Longleftrightarrow TA^*\subseteq B^*T\]
where $A$ and $B$ are normal (not necessarily bounded) and $T\in
B(H)$.

There are different proofs of the Fuglede-Putnam Theorem besides the
first two due to Fuglede and Putnam (e.g. the one in
\cite{Halmos-Fuglede-Putnam}). Perhaps the most elegant proof
 is due to Rosenblum (see \cite{Ros-Fug-Put-1958}). Then came Berberian
\cite{BER}, who noted that the Fuglede's version is actually
equivalent to the Putnam's version. Other proofs which are not
well-spread may be consulted in \cite{Deeds-Fuglede-SHORT-proof} or
\cite{RehArchmath}. See also \cite{Paliogiannis Fuglede}.

There is a particular terminology to the transformation which occurs
in the Fuglede-Putnam Theorem.

\begin{defn}
We say that $T\in B(H)$ \textbf{intertwines}\index{Intertwine} two
operators $A,B$ if $TA\subseteq BT$.
\end{defn}

Accordingly, we may restate the Fuglede-Putnam theorem as follows:
\textit{If an operator intertwines two normal operators, then it
intertwines their adjoints}.

There have been many generalizations of the Fuglede (-Putnam)
Theorem since Fuglede's paper. We note a generalization to the
so-called "spectral operators" by Dunford \cite{Dunford-Spectral
operators Fuglede} (and another proof of the latter by
Radjavi-Rosenthal \cite{Radjavi-Rosenthal Fuglede THM THEIR PAPER}).
See also \cite{FUR}, \cite{Furuta-Fuglede-Subnormal},
\cite{Mecheri-Fuglede-2004}, \cite{Mecheri-et-al BKMS-FUGLEDE},
\cite{MHM.FP}, \cite{Moslehian-Nabavi-FUGLEDE}, \cite{RAJ} and
\cite{SW} (among others). See also
\cite{Mortad-2011-survey-Embry-Fuglede}. For new versions of the
Fuglede-Putnam Theorem involving unbounded operators only, readers
may wish to consult \cite{MHM1},
\cite{Mortad-Fuglede-Putnham-All-unbd},
\cite{Paliogiannis-NEW-proofs-MORTAD} and
\cite{paliogiannis-Fug-Putnam-ALL-UNBD}. An interesting and related
paper is \cite{Jablonski et al 2014}.

Most of these generalizations seem to go into one direction only,
that is, towards relaxing the normality hypothesis whilst there are
still some unexplored territories as regards the very first version.
To get to one main problem of this paper, observe that if $A$ is
self-adjoint (and unbounded), then obviously $BA\subset AB$ implies
that $B^*A\subset AB^*$  for \textit{any} $B\in B(H)$. In
\cite{Jorgensen-1980-JFA}, it was asked whether the assumption of
the self-adjointness of $A$ can be relaxed to requiring only the
closedness of $A$ and imposing the normality of $B$? The referee of
the same reference informed Jorgensen of Fuglede's example (which we
will be recalling below). In the same reference it was shown that
$B^*A\subset AB^*$ if for instance the complement of $\sigma(B)$ is
connected and the interior is empty. Readers might also be
interested in \cite{Dehimi-Mortad-BERBERIAN-BKMS}.

Closely related to what has just been alluded at, the following
conjecture was proposed in \cite{Meziane-Mortad} (it has resisted
solutions for about three years). See Theorem \ref{Fuglede NEW
Theorem 2017} and Proposition \ref{Fuglede BT sub TB* T S.A. QUES}.

\begin{conj}\label{CONJ.} Let $T$ be an operator
(densely defined and closed if necessary) and let $B\in B(H)$ be
normal. Then
\[BT\subset TB^* \Longrightarrow B^*T\subset TB.\]
\end{conj}

What is interesting about this conjecture is the fact that it holds
when $T\in B(H)$ (as we recover the bounded version of the
Fuglede-Putnam Theorem), and as it stands, it is covered by none of
the known (unbounded) generalizations of Fuglede-Putnam Theorem (see
e.g. \cite{Mortad-Fuglede-Putnham-All-unbd},
\cite{paliogiannis-Fug-Putnam-ALL-UNBD} and \cite{STO}).

In this paper, we show that this conjecture is true in case $B$ has
a finite pure point spectrum (Theorem \ref{Fuglede NEW Theorem
2017}). It is, however, not true even if we assume that $A$ is
self-adjoint and $B$ is unitary. In the second part of this paper,
we provide a pair of a closed and self-adjoint (unbounded) operators
which is not intertwined by any (bounded or closed) operator except
the zero operator.

Finally, we refer readers to \cite{tretetr-book-BLOCK} for
properties and results about matrices of unbounded operators which
will be helpful in the sequel. For the general theory of unbounded
operators, readers may consult \cite{Birman-Solomjak-sepctral
theorem of s.a. operators} or \cite{SCHMUDG-book-2012} or
\cite{Weidmann}.

\section{Main Results}

\begin{thm}\label{Fuglede NEW Theorem 2017}
Let $B$ be a bounded normal operator with a finite pure point
spectrum and let $A$ be a closed (possibly unbounded) operator on a
separable complex Hilbert space $H$. Let $f,g:\C\to \C$ be two
continuous functions. Then
\[BA\subset Af(B)\Longrightarrow g(B)A\subset A(g\circ f)(B).\]
\end{thm}

\begin{proof}The hypothesis $BA\subset Af(B)$ clearly gives
\[D(A)=D(BA)\subset D(Af(B))\]
where $D$ stands for domains. Hence
\[D(B^2A)=D[B(BA))]\subset D[B(Af(B))]=D[(BA)f(B)]\subset D[(Af(B))f(B)]=D[A(f(B))^2],\]
and next successively, for any $x\in D(A)$,
\[B^2Ax=B(BA)x=B(Af(B))x=(BA)f(B)x=(Af(B))f(B)x=A(f(B))^2x.\]
Hence $B^2A\subset A(f(B))^2$ and by iteration
\[B^mA\subset A(f(B))^m\]
for any $m\in\N$. Therefore,
\[p(B)A\subset A(p\circ f)(B)\]
for any polynomial $p\in \C(X)$.

By assumption, $B$ has a point spectrum with finitely many distinct
eigenvalues $\lambda_j,j\in\{1,\cdots,n\}$, and corresponding
eigenprojectors $E_j$ adding up to the identity operator $I$, so
$B=\sum_{j=1}^n\lambda_jE_j$ is the spectral representation of $B$.
For the given continuous function $g:\C\to \C$, there exists a
polynomial $p$ such that $p(f(\lambda_j))=g(f(\lambda_j))$. In fact,
for any $k\in\{1,\cdots,n\}$, there is a polynomial $p_k$ with roots
$f(\lambda_j)$, $j\neq k$, and with the value
$p(f(\lambda_k))=(g\circ f)(\lambda_k)$ at $f(\lambda_k)$. Then the
polynomial $p:=\sum_{k=1}^np_k$ has the asserted property. From the
hypothesis $BA\subset Af(B)$, we obtain
\[g(B)A=p(B)A\subset A(p\circ f)(B)=A(g\circ f)(B).\]

\end{proof}

\begin{cor}\label{Fugelde partial coro}With $A$ and $B$ as in Theorem \ref{Fuglede NEW
Theorem 2017}, we have
\[BA\subset AB^*\Longrightarrow B^*A\subset AB.\]
\end{cor}

\begin{proof}
Just apply Theorem \ref{Fuglede NEW Theorem 2017} to the functions
$f,g:z\mapsto \overline{z}$ (so that $g\circ f$ becomes the identity
map on $\C$).
\end{proof}

A similar reasoning applies to establish the following consequence:

\begin{cor}\label{Fugelde partial coro II}With $A$ and $B$ as in Theorem \ref{Fuglede NEW
Theorem 2017}, we likewise have
\[BA\subset AB\Longrightarrow B^*A\subset AB^*.\]
\end{cor}

Using an idea by Berberian, we may generalize this result to the
case of two normal operators whereby we obtain a Fuglede-Putnam
style theorem.

\begin{pro}Let $B$ and $C$ be bounded normal operators with a finite pure point spectrum and
let $A$ be a closed (possibly unbounded) operator on a separable
complex Hilbert space $H$. Then
\[BA\subset AC\Longrightarrow B^*A\subset AC^*.\]
\end{pro}

\begin{proof}Define $\tilde{B}$ on $H\oplus H$ by:
\[\tilde{B}=\left(
              \begin{array}{cc}
                B & 0 \\
                0 & C \\
              \end{array}
            \right)
\]
and let $\tilde{A}=\left(
             \begin{array}{cc}
               0 & A \\
               0 & 0 \\
             \end{array}
           \right)$ with $D(\tilde{A})=H\oplus D(A)$. Since $BA\subset
           AC$, it follows that
           \[\tilde{B}\tilde{A}=\left(
             \begin{array}{cc}
               0 & BA \\
               0 & 0 \\
             \end{array}
           \right)\subset \left(
             \begin{array}{cc}
               0 & AC \\
               0 & 0 \\
             \end{array}
           \right)=
           \tilde{A}\tilde{B}\] for $D(\tilde{B}\tilde{A})=H\oplus
           D(A)\subset H\oplus D(AC)=\tilde{A}\tilde{B}$.

           Now, since $B$ and $C$ are normal, so is
           $\tilde{B}$. Finally, apply Corollary \ref{Fugelde partial coro II} to the
           pair $(\tilde{B},\tilde{A})$ to get
           \[\tilde{B}^*\tilde{A}\subset
           \tilde{A}\tilde{B}^*\] which, upon examining their
           entries, yields the required result.
\end{proof}

\begin{cor}Let $B$ and $C$ be bounded normal operators with a finite pure point spectrum and
let $A$ be a densely defined operator on a separable complex Hilbert
space $H$. Then
\[BA\subset AC\Longrightarrow CA^*\subset A^*B.\]
\end{cor}

\begin{proof}
Merely use the foregoing result, then take adjoints.
\end{proof}

One may wonder whether $BT\subset TB^*$ implies $B^*T\subset TB$ in
the events of the self-adjointness of $T$ and the normality of $B\in
B(H)$? The next example says that this is untrue, thus providing a
counterexample to Conjecture \ref{CONJ.}.

\begin{pro}\label{Fuglede BT sub TB* T S.A. QUES}
There is a unitary $B\in B(H)$ and a self-adjoint $T$ with domain
$D(T)\subset H$ such that $BT\subset TB^*$ but $B^*T\not\subset TB$.
\end{pro}

First, we recall the following example (which appeared in
\cite{FUG-1954-cexp-proper extension}):

\begin{exa}
There exists a unitary $B\in B(H)$ and a closed and symmetric $T$
with domain $D(T)\subset H$ such that $BT\subset TB$ but
$B^*T\not\subset TB^*$.
\end{exa}

Now, we prove Proposition \ref{Fuglede BT sub TB* T S.A. QUES}.

\begin{proof}Consider a unitary
$U\in B(H)$ and a closed $A$ such that $UA\subset AU$ and
$U^*A\not\subset AU^*$. Consider
\[B=\left(
      \begin{array}{cc}
        U & 0 \\
        0 & U^* \\
      \end{array}
    \right)\text{ and } T=\left(
                            \begin{array}{cc}
                              0 & A \\
                              A^* & 0 \\
                            \end{array}
                          \right).
\]
Then $B$ is unitary and $T$ is self-adjoint on $D(A^*)\oplus D(A)$
(thanks to the closedness of $A$). Besides,
\[BT=\left(
         \begin{array}{cc}
           0 & UA \\
           U^*A^* & 0 \\
         \end{array}
       \right)\text{ and }TB^*=\left(
                               \begin{array}{cc}
                                 0 & AU \\
                                 A^*U^* & 0 \\
                               \end{array}
                             \right).\]
Since $UA\subset AU$, it follows by taking adjoints that
$U^*A^*\subset A^*U^*$. Therefore, $BT\subset TB^*$. Since
$U^*A\not\subset AU^*$ is equivalent to $UA^*\not\subset A^*U$, we
may thereby get that
\[B^*T\not\subset TB\]
as $D(B^*T)\not\subset D(TB)$.
\end{proof}

Now, we pass to the second topic of the paper. Fuglede found in
\cite{FUG} a closed operator which did not commute with any bounded
operator except scalar ones (i.e. $\alpha I$ where $\alpha\in\C$).
The next two results lie within the same scope. In addition, they
allow us to establish the uniqueness of the solution of some
particular equations.

\begin{pro}\label{Fugl MINE two interwining A B QUESTION} On some Hilbert space
$H$, there is a self-adjoint operator $A$ and a densely defined
closed operator $B$ such that $TA\subset BT$ (whenever $T\in B(H)$)
implies $T=0$. Also (for the same pair $A$ and $B$), $SB\subset AS$
for any $S\in B(H)$ forces $S=0$.
\end{pro}

\begin{proof}
Let $H=L^2(\R)\oplus L^2(\R)$ and let $A$ be any unbounded
self-adjoint operator with domain $D(A)\subset H$ and let $B$ be a
closed operator such that
\[D(B^2)=D\left({B^*}^2\right)=\{0\}\]
(as in \cite{Dehimi-Mortad-Chernoff}, cf.
\cite{Mortad-Powers-Trivial-domains-CEXAMP}). Let $T\in B(H)$. Then,
clearly
\[TA\subset BT\Longrightarrow TA^2\subset BTA\subset B^2T.\]
Hence
\[D(TA^2)=D(A^2)\subset D(B^2T)=\{x\in H: Tx\in D(B^2)=\{0\}\}=\ker T.\]
Since $A^2$ is densely defined, it follows that
\[H=\overline{D(A^2)}\subset \overline{\ker T}=\ker T\subset H.\]
Hence $\ker T=H$, that is, $T=0$, as required.

Now, we pass to the second part of the question. Plainly,
\[SB\subset AS\Longrightarrow S^*A\subset B^*S^*.\]
As before, we obtain
\[S^*A^2\subset {B^*}^2S^*.\]
Similar arguments as above then yield $S^*=0$ or simply $S=0$, as
needed.
\end{proof}

\begin{rema}
In fact, the first case of the foregoing counterexample may be
beefed up by even allowing $B$ to be also symmetric and semi-bounded
(see e.g. \cite{Weidmann} for the definition of semi-boundedness).
This is based on the famous counterexample by Chernoff in \cite{CH}.
\end{rema}

\begin{pro}
On some Hilbert space $H$, there are two densely defined closed
operators $A$ and $B$ such that $TA\subset BT$ implies $T=0$
whenever $T$ is closed.
\end{pro}

\begin{proof}
Let $H=L^2(\R)\oplus L^2(\R)$ and let $A$ be a densely defined
closed operator with domain $D(A)\subset L^2(\R)\oplus L^2(\R)$ such
that $A^2=0$ on $D(A^2)=D(A)$ (cf.
\cite{Ota-unboudned-nlipotent-idempotent}). An explicit and adapted
example to our case is to consider
\[A=\left(
      \begin{array}{cc}
        0 & T \\
        0 & 0 \\
      \end{array}
    \right)
\]
where $T$ is say an unbounded self-adjoint operator with domain
$D(T)\subset L^2(\R)$. By definition, $D(A)=L^2(\R)\oplus D(T)$.

Then as may easily be seen
\[A^2=\left(
      \begin{array}{cc}
        0 & 0_{D(T)} \\
        0 & 0_{D(T)}  \\
      \end{array}
    \right)\]
with $D(A^2)=D(A)=L^2(\R)\oplus D(T)$. Now, let $B$ be a closed
operator defined on $L^2(\R)\oplus L^2(\R)$ satisfying
$D(B^2)=\{0\}$ (as in \cite{Dehimi-Mortad-Chernoff}).

Now, clearly
\[TA\subset BT\Longrightarrow TA^2 \subset B^2T.\]
But
\[D(TA^2)=\{x\in D(A^2):0\in D(T)\}=D(A)\text{ and }D(B^2T)=\ker T.\]
Hence
\[D(A)\subset \ker T\subset L^2(\R)\oplus L^2(\R)\]
and so upon passing to the closure (w.r.t. $L^2(\R)\oplus L^2(\R)$)
\[\ker T=L^2(\R)\oplus L^2(\R)\]
because $\ker T$ is closed for $T$ is closed. Therefore, $Tx=0$ for
all $x\in D(T)$, i.e. $T\subset 0$. Accordingly, as $T$ is bounded
on $D(T)$ and also closed, then $D(T)$ becomes closed and so
$D(T)=H$, that is, $T=0$ everywhere, as coveted.
\end{proof}

\section{Concluding Remarks and an Open Problem}

It seems noteworthy that easy arguments allow us to show that
$BT=TB^*$ does imply that $B^*T=TB$ when $B$ is unitary even if $T$
is any (unbounded) operator. In other words, the self-adjointness of
$BT$ entails that of $B^*T$ if we further assume that $T$ is
self-adjoint. One may therefore wonder what happens if one assumes
that $B$ is only normal? The problem thus becomes: If $B\in B(H)$ is
normal and if $T$ is (unbounded) self-adjoint, then
\[BT\text{ is self-adjoint}\Longleftrightarrow B^*T\text{ is self-adjoint?}\]

Recall that if $T$ is bounded, then the self-adjointness of $BT$
gives the self-adjointness of $B^*T$ and vice versa. The analogous
question in the case of normality of $BT$ has already a negative
answer as a famous counterexample by Kaplansky shows (see
\cite{Kapl}. Cf. \cite{Benali-Mortad}). Going back to the main
question, observe that a naive counterexample is not available
either. In other words, if $BT$ is closed, then $B^*T$ is
necessarily closed (and conversely). Indeed, the normality of $B$
gives
\[\|B^*Tx\|=\|BTx\|,~\forall x\in D(B^*T)=D(BT)=D(T).\]
Hence, the graph norms of $B^*T$ and $BT$ coincide and hence the
closedness of one implies the closedness of the other. With the
closedness of $B^*T$ at hand, we may try to show that $B^*T$ is
normal and having a real spectrum. But honestly, we just do not know
whether this would lead anywhere or one has to look for
counterexamples?

\end{document}